\documentclass[12pt,reqno]{article}
\usepackage{geometry}
\usepackage{XCharter}
\usepackage[T1]{fontenc}
\usepackage{amsthm}
\usepackage{amssymb}
\usepackage{amsmath}

\usepackage{lipsum}
\usepackage[mathscr]{eucal}
\usepackage[mathlines]{lineno}
\usepackage{latexsym}
\usepackage{xstring}
\usepackage{graphicx,booktabs,multirow}
\usepackage{tikz}
\usepackage{appendix}
\usepackage{yhmath}
\usepackage{float}
\usepackage{color}
\usepackage{subfigure}
\usepackage{rotating}
\usepackage[section]{placeins}
\usepackage{setspace}
\usepackage{pdfpages}
\usepackage{indentfirst} 

\newtheorem{theorem}{Theorem}
\newtheorem{lemma}[theorem]{Lemma}

\newtheorem{problem}[theorem]{Problem}

\newtheorem{conjecture}[theorem]{Conjecture}

\AddToHook{env/proof/begin}{\setcounter{case}{0}} 

\usepackage[colorlinks,
linkcolor=red!60!black,
anchorcolor=blue,
citecolor=blue!60!black
]{hyperref}

\numberwithin{equation}{section}

\usepackage{pdfpages} 
\usepackage{caption}

\usetikzlibrary{backgrounds}
\usetikzlibrary{arrows}
\usetikzlibrary{shapes,shapes.geometric,shapes.misc}

\let\oldbibliography\thebibliography
\renewcommand{\thebibliography}[1]{%
  \oldbibliography{#1}%
  \setlength{\itemsep}{-2pt}%
  \setlength{\baselineskip}{15pt}
  \setlength{\lineskiplimit}{-\maxdimen}
}

\usepackage[colorlinks,
linkcolor=red!60!black,
anchorcolor=blue,
citecolor=blue!60!black
]{hyperref}
\usepackage{url}

\numberwithin{equation}{section}

\usepackage{pdfpages} 
\usepackage{caption}

\usetikzlibrary{backgrounds}
\usetikzlibrary{arrows}
\usetikzlibrary{shapes,shapes.geometric,shapes.misc}
\pgfdeclarelayer{edgelayer}
\pgfdeclarelayer{nodelayer}
\pgfsetlayers{background,edgelayer,nodelayer,main}
\tikzstyle{none}=[inner sep=0mm]
\tikzstyle{bluenode}=[fill=blue, draw=black, shape=circle, minimum size=0.18cm,
inner sep=0pt]
\tikzstyle{blacknode}=[fill=black, draw=black, shape=circle, minimum
size=0.18cm, inner sep=0pt]
\tikzstyle{rednode}=[fill={rgb,255: red,244; green,0; blue,0}, draw=black,
shape=circle, minimum size=0.15cm, inner sep=0.1pt]
\tikzstyle{square}=[draw=black, shape=rectangle, minimum
size=0.15cm, inner sep=0pt]
\tikzstyle{whitenode}=[fill={rgb,255: red,245; green,245; blue,245},
draw=black, shape=circle, minimum size=0.18cm, inner sep=1pt]
\tikzstyle{whitenode_v2}=[fill={rgb,255: red,245; green,245; blue,245},
draw=black, shape=circle, minimum size=0.2cm, inner sep=0.8pt, scale=0.4]
\tikzstyle{greennode_v2}=[fill={rgb,255: red,150; green,245; blue,150},
draw=black, shape=circle, minimum size=0.2cm, inner sep=1pt, scale=0.5]
\tikzstyle{blacknode_v1}=[fill=black, draw=black, shape=circle, minimum
size=0.1cm, inner sep=0pt]
\tikzstyle{square}=[draw=black, shape=rectangle, minimum
size=0.15cm, inner sep=1pt,fill={rgb,255: red,245; green,245;
blue,245}]
\tikzstyle{blue_thick}=[-, line width=0.4mm, draw=blue]
\tikzstyle{blue_bold}=[-, draw=blue, line width=0.4mm]
\tikzstyle{black_bold}=[-, draw=black, line width=0.4mm]
\tikzstyle{blueedge}=[-, draw={rgb,255: red,0; green,0; blue,224}, line
width=0.35mm]
\tikzstyle{blackedge}=[-, draw=black, fill=none, line width=0.3mm]
\tikzstyle{blackedge_thick}=[-, draw=black, line width=0.4mm, fill=none]
\tikzstyle{balck_dash}=[-, dash pattern=on 0.2mm off 0.2mm]
\tikzstyle{rededge}=[-, line width=0.2mm, draw=red]
\tikzstyle{greenedge}=[-, draw={rgb,255: red,9; green,122; blue,43}]
\tikzstyle{rededge_thick}=[-, line width=0.45mm, draw=red]
\tikzstyle{blackedge_opacity}=[-, -, draw={rgb,255: red,91; green,87; blue,84},
fill=none, line width=0.2mm, opacity=0.7]
\tikzstyle{blue_thick}=[-, line width=0.5mm, draw=blue]

\begin{document}

\title{ The maximum number of edges of bipartite 1-planar graphs with 1-disk drawings }

\author{ Guiping Wang\thanks{ Email: wanggpmath@163.com}\\
{\small College of Mathematics and Statistics} \\
{\small Hunan Normal University, Changsha 410081, P.R.China}}
\date{}
\maketitle

\begin{abstract}
A graph is $1$-planar if it admits a drawing in the plane such that each edge is crossed at most once. 
Let $G$ be a bipartite $1$-planar graph with bipartition sets $X$ and $Y$.
A $1$-disk $\mathcal{O}_{X}$  drawing of  $G$  is a $1$-planar drawing such that all vertices of $X$ lie on the boundary of $\mathcal{O}$ 
and all vertices of $Y$ and all edges of $G$ locate in the interior of $\mathcal{O}$, 
where $\mathcal{O}$ is a disk on the plane. 
The concept  was first proposed by Huang, Ouyang and Dong 
when they solved a conjecture about the edge density of bipartite $1$-planar graphs. 
Additionally, they presented a problem of determining the maximum number of edges in a bipartite graph with a $1$-disk $\mathcal{O}_{X}$ drawing. 
In this paper, we solve this problem and prove that every bipartite graph $G$ which has a  $1$-disk $\mathcal{O}_{X}$  drawing has at most $2|V(G)|+|X|-6$ edges.
Moreover, we demonstrate that this upper bound is tight,  in the sense that there
are infinitely many graphs for which this bound is attained exactly.

\bigskip\noindent \textbf{Keywords} edge density $\cdot$ bipartite graph $\cdot$ $1$-planar graph $\cdot$ $1$-disk drawing

\bigskip\noindent \textbf{Mathematics Subject Classification} 05C10 $\cdot$ 05C35 $\cdot$ 05C62

\end{abstract}

\section{Introduction}
All graphs considered in this paper are simple. 
For a graph $G$, let $V(G)$, $E(G)$, $|V(G)|$ and $|E(G)|$ denote the vertex set, the edge set, the order and the size of $G$, respectively.
We assume the reader is familiar with the standard graph theory terminology
(see \cite{JB2008}).

A {\it drawing} of a graph $G=(V,E)$ is a mapping $D$ that assigns to each vertex in $V$ a distinct point in the plane 
and to each edge $uv$ in $E$ a continuous arc connecting $D(u)$ and $D(v)$.
In recent decades, the problem of determining the maximum number of edges in graphs with restricted drawings has received considerable attention.
Especially in the context of non-planar graphs, for example, see \cite{WG2019} for a survey.



A graph is {\it bipartite} if its vertex set can be partitioned into two subsets $X$ and $Y$ such that every edge joins a vertex in $X$ and a vertex in $Y$.
A graph is  {\it planar} if it has a drawing such that no edges cross each other. 
A graph is  {\it $1$-planar} if it has a drawing such that each edge is crossed at most once.

It is well known that every planar graph with $n(n\geq3)$ vertices has at most $3n-6$ edges. 
A bipartite planar graph with $n (n\geq3)$ vertices has at most $2n-4$ edges.
Any 1-planar graph with $n(n\geq3)$ vertices has at most $4n-8$ edges \cite{HB1984,IF2007,JP1997}.
In particular, 
Karpov \cite{DK2014} proved that every bipartite $1$-planar graph has at most $3n-8$ edges for even $n\neq6$ and at most $3n-9$ for odd $n$ and for $n=6$.
Note that Karpov’s upper bound on the size of a bipartite 1-planar graph is determined by its order.
When the sizes of the bipartition sets of a bipartite 1-planar graph are given, Czap, Przyby{\l}o and \v{S}krabul’áková \cite{JJ2016} obtained that 
if $G$ is a bipartite 1-planar graph with bipartition sets of sizes $|X|$ and $|Y|$ satisfying $2\leq |X| \leq |Y|$, 
then $|E(G)|\leq2|V(G)|+6|X|-16$ and the bound is tight for $|X|=2$.
For each pair of integers $|X|$ and $|Y|$ with $|X|\geq 3$ and $|Y|\geq 6|X|-12$, the authors in \cite{JJ2016}
constructed a bipartite 1-planar graph $G$ with bipartition sets of sizes $|X|$ and $|Y|$ and $|E(G)|\geq 2|V(G)|+4|X|-12$. 
Moreover, they believed this lower bound is optimal for such graphs and thus put forward the following conjecture.

\begin{conjecture}[\cite{JJ2016}]\label{conjecture}
For any integers $|X|$ and $|Y|$ with $|X|\geq3$ and $|Y|\geq 6|X|-12$, if $G$ is a bipartite 1-planar graph with bipartition sets of sizes $|X|$ and $|Y|$, 
then $|E(G)|\leq2|V(G)|+4|X|-12$.
\end{conjecture}

In 2021, Huang, Ouyang and Dong \cite{YH2021} proved the following theorem 
which confirms  Conjecture \ref{conjecture}  under a weaker condition.

\begin{theorem}[\cite{YH2021}]\label{tightbi}
Let $G$ be a bipartite 1-planar graph with bipartition sets of sizes $|X|$ and $|Y|$, where $2\leq |X| \leq |Y|$. 
Then $|E(G)|\leq2|V(G)|+4|X|-12$. Moreover, this bound is tight.
\end{theorem}

Especially, in \cite{YH2021}, the authors proposed an important definition when they solved the conjecture.
A {\it $1$-disk $\mathcal{O}_{X}$ drawing} of a bipartite graph $G$ with bipartition sets $X$ and $Y$ is a $1$-planar drawing  
such that all vertices of $X$ lie on the boundary of $\mathcal{O}$ and all vertices of $Y$ and all edges of $G$ lie in the interior of $\mathcal{O}$, 
where $\mathcal{O}$ is a disk on the plane up to a homeomorphism of the plane. 
For the upper bound on the size of a bipartite graph which has a 1-disk $\mathcal{O}_{X}$ drawing and  $|X|=3$, 
the authors in \cite{YH2021} presented the following lemma.

\begin{lemma}[\cite{YH2021}]\label{xis3}
Let $G$ be a bipartite graph with bipartition sets $X$ and $Y$ which has a 1-disk $\mathcal{O}_{X}$ drawing.
If $|X|=3$, then $|E(G)|\leq 2|Y|+3$.
\end{lemma}

For bipartite graphs without the condition $|X|=3$, they posed the following general problem.

\begin{problem}[\cite{YH2021}]\label{oxdrawing}
Let $G$ be a bipartite graph with bipartition sets $X$ and $Y$ which has a 1-disk $\mathcal{O}_{X}$ drawing.
Is it true that $|E(G)|\leq 2|Y|+5|X|/3-2$?
\end{problem}

In this paper, we consider this problem and improve the result of Problem \ref{oxdrawing} as follows.

\begin{theorem}\label{tight upper bound}
Let $G$ be a bipartite graph that has bipartition sets $X$ and $Y$ with $2\leq |X| \leq |Y|$. 
If $G$ has a 1-disk $\mathcal{O}_{X}$ drawing,  
then $|E(G)|\leq 2|V(G)|+|X|-6$.
\end{theorem}

The following theorem will show that the upper bound in Theorem \ref{tight upper bound} is best possible,
in the sense that there are infinitely many graphs for which this bound is attained exactly.

\begin{theorem}\label{extremalgraph}
Let $|X|$ and $|Y|$ be any integers such that $|X|\geq3$ and $ |Y|\geq 3|X|-6$. 
Then there exists a bipartite graph $G$ with bipartition sets of sizes $|X|$ and $|Y|$ which  has a 1-disk $\mathcal{O}_{X}$ drawing 
and $|E(G)|= 2|V(G)|+|X|-6$.
\end{theorem}

%
%


Observe that the upper bound in Theorem \ref{tight upper bound} is $|E(G)|\leq 2|V(G)|+|X|-6=2|Y|+3|X|-6$, 
which provides a better bound for $|E(G)|$ than  Problem \ref{oxdrawing}.
On the other hand, the upper bound in Theorem \ref{tight upper bound} is about the size of a bipartite graph with a restricted drawing. 
Indeed, there are many analogous results, for example,
for graphs with additional constraints on vertex positions in the drawing,
Didimo \cite{WD2013} showed that a 1-planar graph has at most $2.5n-4$ edges if it has a $1$-planar drawing such that all vertices lie on a longest cycle,
and a bipartite graph has at most $1.5n-2$ edges if it has a $1$-planar drawing such that its vertices lie on two distinct horizontal layers 
(for definitions, see \cite{WD2013}  for details).
For more related results on bipartite graphs, one can refer to Table $1$ in \cite{PA2018}.

\section{ Proof of the Main Theorems}\label{proof}

We now present the proofs of our main results. 

\begin{proof}[Proof of Theorem \ref{tight upper bound}]
Let $D$ be a 1-disk $\mathcal{O}_{X}$ drawing of  $G$, where $|X|$ vertices lie on the boundary of $\mathcal{O}$ 
and all vertices of $Y$ and all edges of $G$ lie in the interior of $\mathcal{O}$. 
Note that $D$ has an unbounded region $f$ whose boundary contains all vertices of $X$.
Now, we perform a spherical mapping $\pi$: $D\rightarrow D'$,
such that all vertices in $X$ of $D'$ lie on the boundary of $\mathcal{O}$ in the same way as in $D$,
and all vertices of $Y$ and all edges of $D'$ lie in the exterior of $\mathcal{O}$. 
Clearly,  $D'$ is a 1-planar drawing of $G$ and  there is a bounded region $f'$ in $D'$ which corresponding to the unbounded region $f$ in $D$.
Next, we place $D$ into the region $f'$ of $D'$ and then identify the boundaries of $f$ and $f'$ to obtain a new graph $G^{*}$.
Clearly, $G^{*}$ is a bipartite graph which has a $1$-planar drawing with $|V(G^{*})|=|X|+2|Y|$ and $|E(G^{*})|=2|E(G)|$.
Then, by Theorem \ref{tightbi}, we have
\begin{center}
$|E(G^{*})|\leq 2 |V(G^{*})|+4|X|-12=2(|X|+2|Y|)+4|X|-12=6|X|+4|Y|-12$, 
\end{center}
implying that,
\begin{center}
$2|E(G)|\leq 6|X|+4|Y|-12$ 
\end{center}
that is,
\begin{align}
|E(G)| 
\nonumber
&\leq 3|X|+2|Y|-6  \\
\nonumber
&= 2(|X|+|Y|)+|X|-6  \\
\nonumber
&=2|V(G)|+|X|-6.
\end{align}
Hence, the theorem holds.
\end{proof}

A graph is {\it outerplanar} if it has a planar drawing such that all its vertices lies on the boundary of its outer (unbounded) face. 
An outerplanar graph $G$ is {\it maximal} if $G+uv$ is not outerplanar for any two nonadjacent vertices $u$ and $v$ of $G$. 
Note that any outerplanar graph can be extended to a maximal outerplanar graph with exactly $2n-3$ edges by triangulating the faces other than the outer face. 
We now proceed with our construction of extremal graphs to prove the bound in Theorem \ref{tight upper bound} is tight.


\begin{figure}[h!]
\centering
\begin{tikzpicture}[scale=0.7]
	\begin{pgfonlayer}{nodelayer}
		\node [style=blacknode] (0) at (0, 3) {};
		\node [style=none] (1) at (0, -3) {};
		\node [style=none] (2) at (-3, 0) {};
		\node [style=none] (3) at (3, 0) {};
		\node [style=blacknode] (4) at (-2.45, -1.775) {};
		\node [style=blacknode] (5) at (2.45, -1.8) {};
		\node [style=bluenode] (6) at (0, 0.5) {};
		\node [style=bluenode] (7) at (-0.5, -0.5) {};
		\node [style=bluenode] (8) at (0.5, -0.5) {};
		\node [style=none] (9) at (0, -4) {};
	\end{pgfonlayer}
	\begin{pgfonlayer}{edgelayer}
		\draw [style={balck_dash}, bend left=45] (0) to (3.center);
		\draw [style={balck_dash}, bend right=45] (0) to (2.center);
		\draw [style={balck_dash}, bend right=45] (2.center) to (1.center);
		\draw [style={balck_dash}, bend right=45] (1.center) to (3.center);
		\draw [style=blueedge] (6) to (0);
		\draw [style=blueedge] (4) to (6);
		\draw [style=blueedge] (6) to (5);
		\draw [style=blueedge] (0) to (7);
		\draw [style=blueedge] (4) to (7);
		\draw [style=blueedge] (7) to (5);
		\draw [style=blueedge] (8) to (5);
		\draw [style=blueedge] (8) to (0);
		\draw [style=blueedge] (8) to (4);
	\end{pgfonlayer}
\end{tikzpicture}
 \caption{The configuration $B_{3}$.}
 \label{B3}
\end{figure}
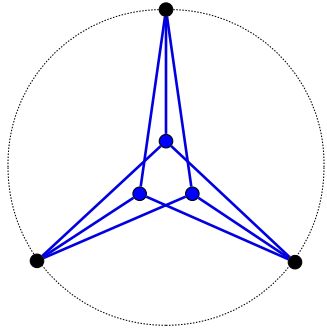

\begin{proof}[Proof of Theorem \ref{extremalgraph}] 
We shall use  maximal outerplanar graphs to construct a class of bipartite graphs with bipartition sets of sizes $|X|$ and $|Y|$ 
such that $|E(G)|=2|V(G)|+|X|-6$ and $G$ has a 1-disk $\mathcal{O}_{X}$ drawing.
First assume that $|Y|=3(|X|-2)$.


Let $H$ be a maximal outerplanar graph with $|X|$ vertices. We color the vertices and edges in $H$ by black.
By Euler's formula, there are $|X|-2$ triangular inner faces  in $H$.
Let $H^{*}$ be the graph obtained from $H$ by inserting a configuration $B_{3}$ depicted in Fig. \ref{B3} into each of its triangular inner faces.
And we color the newly added vertices and edges of $H^{*}$ by blue, as shown  in Fig. \ref{extremal graph} (a).
Let $G$ be the graph obtained from $H^{*}$ by removing all black edges of $H$.
See Fig. \ref{extremal graph} (b) for example.
Note that blue vertices are independent and each blue vertex is adjacent to three black vertices.
Then there are three blue vertices and nine blue edges are added into each  triangular inner faces  of $H$.
Thus there are $|X|$ black vertices and $3(|X|-2)$ blue vertices in $G$,
i.e., $G$ is a bipartite graph with bipartition sets of sizes $|X|$ and $|Y|=3(|X|-2)$.  
From our construction, it is not difficult to see that $G$ has a 1-disk $\mathcal{O}_{X}$ drawing $D$.
And it can be checked that  $|V(G)|=|X|+|Y|=|X|+3(|X|-2)=4|X|-6$
and $|E(G)|=9(|X|-2)= 9|X|-18=3|X|+2\times (3|X|-6)-6=3|X|+2|Y|-6=2|V(G)|+|X|-6$, as desired.

Now further suppose that $|Y|=3(|X|-2)+t$ for some $t\geq1$. 

In this case,   let $G$ be the bipartite $1$-planar graph with bipartition sets of sizes $|X|$ and $3(|X|-2)$ that is obtained as above. 
  Let $D$  be a 1-disk $\mathcal{O}_{X}$ drawing of $G$.
We add $t$ vertices into an arbitrary  region of $D$ that is incident with two black vertices 
and join each of the $t$ vertices (without any edge crossing) to the two black vertices.
In such a way we obtain a new bipartite $1$-planar graph $G'$ with bipartition sets of sizes $|X|$ and $|Y|=3(|X|-2)+t$. 
Clearly $G'$ has a $1$-disk $\mathcal{O}_{X}$ drawing. 
And we can verify that $|V(G')|=|X|+|Y|=|X|+3(|X|-2)+t=4|X|-6+t$  
and $|E(G')|=9(|X|-2)+2t=9|X|-18+2t=3|X|+2(3|X|-6+t)-6=3|X|+2|Y|-6=2|V(G')|+|X|-6$. 
Consequently, the proof is completed.
\end{proof}

\begin{figure}[h!]
\centering
\begin{tikzpicture}[scale=0.7]
	\begin{pgfonlayer}{nodelayer}
		\node [style=blacknode] (4) at (5, 3) {};
		\node [style=none] (5) at (5, -3) {};
		\node [style={blacknode_v1}] (6) at (2, 0) {};
		\node [style={blacknode_v1}] (7) at (8, 0) {};
		\node [style=none] (13) at (3.25, 2.475) {};
		\node [style=none] (14) at (6.75, 2.475) {};
		\node [style=blacknode] (15) at (2.75, -2) {};
		\node [style=blacknode] (16) at (7.25, -2) {};
		\node [style={blacknode_v1}] (17) at (2.175, 1) {};
		\node [style={blacknode_v1}] (18) at (2.175, -1) {};
		\node [style={blacknode_v1}] (19) at (7.75, 1.25) {};
		\node [style={blacknode_v1}] (20) at (7.825, -1.025) {};
		\node [style=bluenode] (21) at (5, 0.5) {};
		\node [style=bluenode] (22) at (4.5, -0.75) {};
		\node [style=bluenode] (23) at (5.5, -0.75) {};
		\node [style=none] (24) at (2.75, 1.25) {};
		\node [style=none] (25) at (4, 2.25) {};
		\node [style=none] (26) at (2.5, 0.25) {};
		\node [style=none] (27) at (3.75, 1.75) {};
		\node [style=none] (28) at (2.75, -0.75) {};
		\node [style=none] (29) at (3.75, 1) {};
		\node [style=none] (30) at (7, 1.5) {};
		\node [style=none] (31) at (5.75, 2.5) {};
		\node [style=none] (32) at (7.25, 0.5) {};
		\node [style=none] (33) at (6, 2) {};
		\node [style=none] (34) at (7.25, -0.5) {};
		\node [style=none] (35) at (6.25, 1.25) {};
		\node [style=none] (36) at (5, 3.5) {};
		\node [style=none] (37) at (7.25, 3) {};
		\node [style=none] (38) at (2.75, 3) {};
		\node [style=none] (39) at (2.5, -2.5) {};
		\node [style=none] (40) at (7.75, -2.5) {};
		\node [style=none] (42) at (5, -4) {$(b)$};
		\node [style=blacknode] (43) at (-5, 3) {};
		\node [style=none] (44) at (-5, -3) {};
		\node [style={blacknode_v1}] (45) at (-8, 0) {};
		\node [style={blacknode_v1}] (46) at (-2, 0) {};
		\node [style=none] (47) at (-6.75, 2.475) {};
		\node [style=none] (48) at (-3.25, 2.475) {};
		\node [style=blacknode] (49) at (-7.25, -2) {};
		\node [style=blacknode] (50) at (-2.75, -2) {};
		\node [style={blacknode_v1}] (51) at (-7.825, 1) {};
		\node [style={blacknode_v1}] (52) at (-7.825, -1) {};
		\node [style={blacknode_v1}] (53) at (-2.25, 1.25) {};
		\node [style={blacknode_v1}] (54) at (-2.175, -1.025) {};
		\node [style=bluenode] (55) at (-5, 0.5) {};
		\node [style=bluenode] (56) at (-5.5, -0.75) {};
		\node [style=bluenode] (57) at (-4.5, -0.75) {};
		\node [style=none] (58) at (-7.25, 1.25) {};
		\node [style=none] (59) at (-6, 2.25) {};
		\node [style=none] (60) at (-7.5, 0.25) {};
		\node [style=none] (61) at (-6.25, 1.75) {};
		\node [style=none] (62) at (-7.25, -0.75) {};
		\node [style=none] (63) at (-6.25, 1) {};
		\node [style=none] (64) at (-3, 1.5) {};
		\node [style=none] (65) at (-4.25, 2.5) {};
		\node [style=none] (66) at (-2.75, 0.5) {};
		\node [style=none] (67) at (-4, 2) {};
		\node [style=none] (68) at (-2.75, -0.5) {};
		\node [style=none] (69) at (-3.75, 1.25) {};
		\node [style=none] (70) at (-5, 3.5) {};
		\node [style=none] (71) at (-2.75, 3) {};
		\node [style=none] (72) at (-7.25, 3) {};
		\node [style=none] (73) at (-7.5, -2.5) {};
		\node [style=none] (74) at (-2.25, -2.5) {};
		\node [style=none] (75) at (-5, -4) {$(a)$};
	\end{pgfonlayer}
	\begin{pgfonlayer}{edgelayer}
		\draw [style={balck_dash}, bend left=45] (6) to (4);
		\draw [style={balck_dash}, bend left=45] (4) to (7);
		\draw [style={balck_dash}, bend right=45] (6) to (5.center);
		\draw [style={balck_dash}, bend right=45] (5.center) to (7);
		\draw [style={balck_dash}] (4) to (15);
		\draw [style={balck_dash}] (4) to (16);
		\draw [style=blueedge] (21) to (4);
		\draw [style=blueedge] (21) to (15);
		\draw [style=blueedge] (21) to (16);
		\draw [style=blueedge] (22) to (15);
		\draw [style=blueedge] (22) to (4);
		\draw [style=blueedge] (22) to (16);
		\draw [style=blueedge] (4) to (23);
		\draw [style=blueedge] (23) to (15);
		\draw [style=blueedge] (23) to (16);
		\draw [style={balck_dash}] (24.center) to (25.center);
		\draw [style={balck_dash}] (26.center) to (27.center);
		\draw [style={balck_dash}] (28.center) to (29.center);
		\draw [style={balck_dash}] (30.center) to (31.center);
		\draw [style={balck_dash}] (32.center) to (33.center);
		\draw [style={balck_dash}] (34.center) to (35.center);
		\draw [style=blackedge, bend left=45] (45) to (43);
		\draw [style=blackedge, bend left=45] (43) to (46);
		\draw [style=blackedge, bend right=45] (45) to (44.center);
		\draw [style=blackedge, bend right=45] (44.center) to (46);
		\draw [style=blackedge] (43) to (49);
		\draw [style=blackedge] (43) to (50);
		\draw [style=blueedge] (55) to (43);
		\draw [style=blueedge] (55) to (49);
		\draw [style=blueedge] (55) to (50);
		\draw [style=blueedge] (56) to (49);
		\draw [style=blueedge] (56) to (43);
		\draw [style=blueedge] (56) to (50);
		\draw [style=blueedge] (43) to (57);
		\draw [style=blueedge] (57) to (49);
		\draw [style=blueedge] (57) to (50);
		\draw [style={balck_dash}] (58.center) to (59.center);
		\draw [style={balck_dash}] (60.center) to (61.center);
		\draw [style={balck_dash}] (62.center) to (63.center);
		\draw [style={balck_dash}] (64.center) to (65.center);
		\draw [style={balck_dash}] (66.center) to (67.center);
		\draw [style={balck_dash}] (68.center) to (69.center);
	\end{pgfonlayer}
\end{tikzpicture}
 \caption{$(a)$ The graph $H^{*}$. $(b)$ The bipartite graph $G$ (dashed lines do not exist in $G$).}
 \label{extremal graph}

\end{figure}
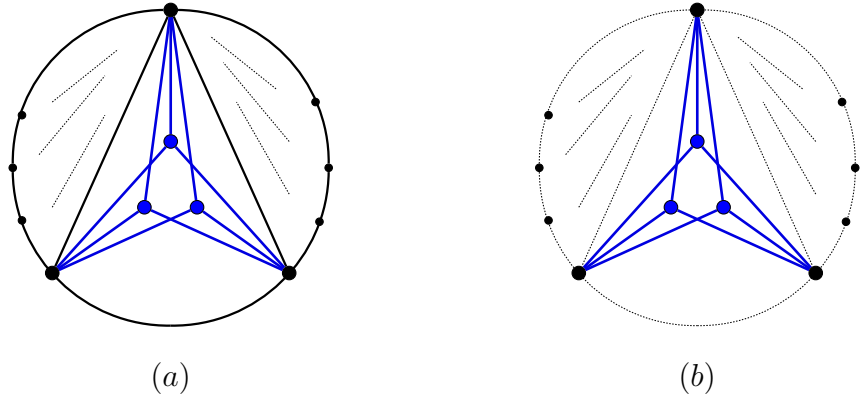

Note that the extremal graphs constructed in Theorem \ref{extremalgraph} exist  for any two integers $|X|$ and $|Y|$  satisfying  $|X|\geq3$ and $|Y|\geq 3|X|-6$, 
rather than for all integers $|X|$ and $|Y|$ within the range of $2\leq |X| \leq |Y|$. 
When $|X|=2$ and $|Y|\geq|X|$, we have $|E(G)|\leq 2|Y|$, and Theorem \ref{tight upper bound} remains valid in this case.
Moreover, for bipartite graphs admitting a 1-disk $\mathcal{O}_{X}$ drawing with $3\leq|X|$ and $|X|\leq |Y | \leq 3|X|-7 $,
the number of edges may be strictly below the bound established in Theorem \ref{tight upper bound}.  
These cases have not been fully characterized, and thus the following problem remains open.
\begin{problem}
Let $G$ be a bipartite graph with bipartition sets $X$ and $Y$ which has a 1-disk $\mathcal{O}_{X}$ drawing.  
If $3\leq|X|$ and $|X|\leq |Y | \leq 3|X|-7 $, determine constants a, b and c such that $|E(G)|\leq a|V(G)|+b|X|+c$.
\end{problem}

\bigskip\noindent \textbf{Data Availability}  There was no data used in the preparation of this manuscript.

\section*{Declarations}
\addcontentsline{toc}{section}{Declarations}
\bigskip\noindent \textbf{Conflict of Interest} The authors declare that they have no known competing financial
interests or personal relationships that could have appeared to influence
the work reported in this paper.

\section*{Acknowledgement}
\bigskip\noindent The author would like  to express their gratitude to the anonymous reviewers for their insightful comments and suggestions, 
which have significantly contributed to enhancing the quality of this paper.

\bibliographystyle{siam}
\bibliography{references}

\end{document}